\documentclass[fleqn
]{article}
\usepackage{amsmath,amssymb,amsthm,xcolor,framed}
\usepackage[english]{babel}

\usepackage[title]{appendix}

\usepackage{graphicx}

\numberwithin{equation}{section}

\usepackage[bookmarks]{hyperref}

\newcommand{\R}{{\mathbb R}}

\newcommand{\Z}{{\mathbb Z}}
\newcommand{\T}{{\mathbb T}}
\renewcommand{\L}{{\mathcal L}}
\newcommand{\e}{\epsilon}
\newcommand{\BV}{{\mathrm{BV}}}
\newcommand{\TV}{{\mathrm{TV}\,}}
\DeclareMathOperator{\dv}{div}

\DeclareMathOperator{\supp}{supp}

\def\Xint#1{\mathchoice
	{\XXint\displaystyle\textstyle{#1}}%
	{\XXint\textstyle\scriptstyle{#1}}%
	{\XXint\scriptstyle\scriptscriptstyle{#1}}%
	{\XXint\scriptscriptstyle\scriptscriptstyle{#1}}%
	\!\int}

\def\XXint#1#2#3{{\setbox0=\hbox{$#1{#2#3}{\int}$}
		\vcenter{\hbox{$#2#3$}}\kern-.5\wd0}}

\newtheorem{thm}{Theorem}[section]
\newtheorem{prop}[thm]{Proposition}
\newtheorem{lem}[thm]{Lemma}

\newtheorem{definition}{Definition}
\theoremstyle{definition}

\newtheorem*{rem*}{Remark}

\title{Rectifiability of entropy productions for weak solutions of the 2D eikonal equation with supercritical regularity}
\date{\today}
\author{Xavier Lamy\thanks{Institut de Math\'ematiques de Toulouse; UMR 5219, Universit\'e de Toulouse; CNRS, UPS IMT, F-31062 Toulouse Cedex 9, France. Email: xlamy@math.univ-toulouse.fr.}
\thanks{Institut Universitaire de France (IUF)}
 \qquad  Elio Marconi\thanks{Dipartimento di Matematica 'Tullio Levi Civita', Universit\`a di Padova, via Trieste 63, 35121 Padova (PD), Italy. Email: elio.marconi@unipd.it.}}

\begin{document}

\maketitle

\begin{abstract}
%
%
Weak solutions $m\colon\Omega\subset\R^2\to\R^2$ of the eikonal equation
\begin{align*}
|m|=1\text{ a.e. and }\mathrm{div}\: m =0\,,
\end{align*}
arise naturally as sharp interface limits of bounded energy configurations in various physically motivated models, including the Aviles-Giga energy.
The distributions $\mu_\Phi=\mathrm{div}\,\Phi(m)$, defined for a class of smooth vector fields $\Phi$ called entropies,
carry information about singularities and energy cost.
If these entropy productions are Radon measures, a long-standing conjecture predicts that
 they must be concentrated on the 1-rectifiable jump set of $m$
 -- as they do if $m$ has bounded variation (BV) thanks to the chain rule. 
We establish this concentration property,
for a large class of entropies,
 under the Besov regularity assumption 
\begin{align*}
m\in B^{1/p}_{p,\infty} \quad
\Leftrightarrow
\quad
\sup_{h\in \mathbb R^2\setminus\lbrace 0\rbrace} \frac{\|m(\cdot +h)-m\|_{L^p
}
}{|h|^{1/p}} <\infty\,,
\end{align*}
for any $1\leq p<3$, thus going
 well beyond the BV setting ($p=1$) and leaving only the borderline case $p=3$ open.
\end{abstract}

\section{Introduction}

For an open set $\Omega\subset\R^2$, 
we consider weak solutions $m\colon\Omega\to\R^2$ of the eikonal equation
\begin{align}\label{eq:eik}
|m|=1\text{ a.e. in }\Omega,\quad \dv m =0\text{ in }\mathcal D'(\Omega).
\end{align}
If $\Omega$ is simply connected, this is equivalent to the existence of a Lipschitz function $u\colon\Omega\to \R$ satisfying
 $m=i\nabla u$ and  
 \begin{align*}
 |\nabla u|=1 \text{ a.e. in }\Omega\,,
 \end{align*}
 which is classically referred to as the eikonal equation.

If $m\colon\Omega\to\R^2$ is a $C^1$ solution of the eikonal equation \eqref{eq:eik}
then the chain rule provides a whole family of conservation laws:
$\dv\Phi(m)=0$ for all 
$\Phi\in C^1(\mathbb S^1;\R^2)$
 such that 
 $\partial_\theta\Phi(e^{i\theta})\cdot e^{i\theta} =0$ for all $\theta\in\R$.
 
For a general weak solution $m\colon\Omega\to\R^2$ of the eikonal equation \eqref{eq:eik},
 the distributions $\dv\Phi(m)$ 
 may not be zero,
 and they carry information on how singular $m$ is.
They were first introduced in the context of the eikonal equation in \cite{dkmo01}, 
and called entropy productions by analogy with the theory of scalar conservation laws.
We denote by
\begin{align}
\mathrm{ENT}=
\big\lbrace
\Phi\in C^{1,1}(\mathbb S^1;\R^2)\colon\frac{d}{d\theta}\Phi(e^{i\theta})\cdot e^{i\theta}=0\;\forall\theta\in\R\big\rbrace\,,
\end{align}
the set of all $C^{1,1}$ entropies.

Weak solutions $m$ of the eikonal equation \eqref{eq:eik}
 whose entropy productions $\dv\Phi(m)$ are finite Radon measures play an important role in the theory of the Aviles-Giga energy \cite{avilesgiga99}. 
The structure of these finite-entropy solutions
 is not fully understood,
 but it is known that they share with functions of bounded variation (BV) several fine properties.
Note that, if $m\in BV(\Omega;\R^2)$ is a weak solution of \eqref{eq:eik}, 
then by the BV chain rule its entropy productions are measures concentrated on the $\mathcal H^1$-rectifiable jump set $J_m$.
For a general finite-entropy solution, denote by
$\nu$ the supremum measure
\begin{align*}
\nu =\bigvee_{\Phi\in\mathrm{ENT},\,\|\Phi\|_{C^{1,1}}\leq 1}|\dv\Phi(m)|\,.
\end{align*}
 In \cite{ODL03} the authors prove that
 the jump set 
\begin{equation}\label{e:defJ}
J_m \doteq \left\{  x \in \Omega : \limsup_{r\to 0}\frac{\nu(B_r(x))}{r}>0 \right\}\,,
\end{equation}
is $\mathcal H^1$-rectifiable and
 $m$ admits left and right $L^1$ traces $\mathcal H^1$-a.e. along $J_m$.
According to a long-standing conjecture on the Aviles-Giga energy \cite{avilesgiga99},
entropy productions should be concentrated on that jump set.

Among all weak solutions of \eqref{eq:eik}, 
the finite-entropy solutions can be characterized, at least locally, in terms of Besov $B^{1/3}_{3,\infty}$ regularity \cite{GL20}.
For $s\in (0,1)$ and $p\geq 1$, a map $m\in L^p(\Omega)$ has 
the Besov regularity $B^{s}_{p,\infty}$ if and only if the seminorm
\begin{align*}
|m|_{B^{s}_{p,\infty}} = \sup_{h\in\R^2\setminus\lbrace 0\rbrace} 
\frac{1}{|h|^s}\|m(\cdot +h)-m\|_{L^p(\Omega\cap(\Omega -h))}\,,
\end{align*}
is finite \cite[\textsection{2.5.12}]{triebel83}.
Between the spaces $BV(\Omega;\mathbb S^1)$ and $B^{1/3}_{3,\infty}(\Omega;\mathbb S^1)$ lies the intermediate scale of spaces
$B^{1/p}_{p,\infty}(\Omega;\mathbb S^1)$, for $1<p<3$.
We prove that the concentration conjecture is true for solutions of \eqref{eq:eik} with that intermediate regularity,
and for entropies in the class
\begin{align}\label{eq:tildeENT}
\widetilde{\mathrm{ENT}} \doteq \left\{  \Phi \in \mathrm{ENT}: \frac{d}{d\theta}\Phi(-e^{i\theta})= - \frac{d}{d\theta}\Phi(e^{i\theta})
\;\forall\theta\in\R\right\},
\end{align}
which corresponds to odd entropies plus constants.
This restriction is due to the same technical reasons as in \cite{marconi21ellipse} (where this class of entropies is denoted by $\mathcal E_\pi$).
The fundamental entropies introduced in \cite{JK00} to establish a sharp lower bound for the Aviles-Giga energy (see also \cite{ADM}) are odd, and therefore covered by our result.

\begin{thm}\label{t:rectif}
Let $\Omega\subset\R^2$ a bounded open set,  $m\colon \Omega\to\mathbb R^2$
a weak solution of the eikonal equation
 \eqref{eq:eik}, 
and assume that
 $m\in B^{  1/p}_{p,\infty}(\Omega)$ for some $p\in [1,3)$.
 Then the entropy productions of $m$ are 1-rectifiable,
 \begin{align*}
 \dv\Phi(m) = \mathbf n\cdot \big(\Phi(m^+)-\Phi(m^-) \big) \,\mathcal H^1 \llcorner J_m\,,
 \end{align*}
 for all $\Phi\in\widetilde{\mathrm{ENT}}$.
\end{thm}

\medskip

To describe the ideas behind Theorem~\ref{t:rectif},
let us consider first a solution $m\in \BV(\Omega;\mathbb S^1)$ of \eqref{eq:eik}.
Then, by the $BV$ chain rule we have
\[
|\dv \Phi(m)| \le C_\Phi |m^+-m^-|^3d \mathcal H^1\llcorner J_m.
\] 
for every entropy $\Phi \in \mathrm{ENT}$.
Moreover, for any
$p\in [1,3)$ we have
\[
\int_{J_m}|m^+-m^-|^p d\mathcal H^1 \le \|m\|^p_{B^{1/p}_{p,\infty}}\,.
\]
As a consequence, the contribution of jumps smaller than
a threshold $\delta>0$ is controlled by
\begin{align}\label{e:main}
&
|\dv \Phi(m)| (\Omega \setminus  (J_m \cap  \{|m^+-m^-|\ge \delta\}) 
\nonumber
\\
& \leq  C_\Phi \int_{J_m \cap \{|m^+-m^-|\le \delta\}} |m^+-m^-|^3 d \mathcal H^1 
\leq  C_\Phi\delta^{3-p} \|m\|^p_{B^{1/p}_{p,\infty}}.
\end{align}
Here we assumed that $m\in BV(\Omega;\mathbb S^1)$
to ensure that $\dv \Phi(m)$ is concentrated on $J_m$,
 but the estimate does not depend on the total variation of $m$.
Moreover, 
a structure result 
 proved in \cite{marconi21ellipse}
about the continuous part of the entropy production
 allows to interpret it as being generated by infinitesimally small jumps.
It is therefore natural to conjecture that the estimate \eqref{e:main} should be true
for solutions $m$ not necessarily of bounded variation.
We prove indeed a similar estimate in
 Proposition \ref{p:estim_small_jumps}, 
 and the main result then follows by letting $\delta\to 0$.

To technically implement these ideas, 
we actually have to argue along trajectories of a Lagrangian representation of $m$,
also introduced in \cite{marconi21ellipse}, 
and which can only provide information on entropies in \eqref{eq:tildeENT}.
The control on the continuous part of entropy production 
at the level of these Lagrangian trajectories
is obtained by using a singular family of entropies,
whose entropy productions are uniformly bounded thanks
to the supercritical Besov regularity assumption.

\medskip

\section{Entropy productions and Besov regularity}\label{s:ent_prod}

In this section we prove that the supercritical Besov regularity assumption provides uniform control
 on entropy productions over families of entropies which 
are unbounded in $C^{1,1}$.

\begin{prop}\label{p:ent_prod}
Let $\Omega\subset\R^2$ a bounded open set,  $m\colon \Omega\to\mathbb S^1$ with $\dv m=0$ and assume that
 $m\in B^{  1/(2+\alpha)}_{2+\alpha,\infty}(\Omega)$ for some $\alpha\in (0,1)$.
 Then we have
 \begin{align*}
 \bigg(\bigvee_{\|\Phi\|_{C^{1,\alpha}} \leq 1} |\dv \Phi(m)|\bigg)(\Omega) \leq C \|m\|^{2+\alpha}_{B^{1/(2+\alpha)}_{2+\alpha,\infty}(\Omega)}\,,
 \end{align*}
 for some $C=C(\alpha)>0$, where the supremum of measure is taken over all $\Phi\in \mathrm{ENT}$ such that $\|\Phi\|_{C^{1,\alpha}} \leq 1$.
\end{prop}

Proposition~\ref{p:ent_prod} can be interpreted
 as an interpolation between the estimates
\begin{align*}
\bigg(
\bigvee_{\|\Phi\|_{C^{1,1}} \leq 1} |\dv \Phi(m)|\bigg)(\Omega) 
&
\lesssim
\|m\|^{3}_{B^{1/3}_{3,\infty}(\Omega)}\,,
\\
\bigg(\bigvee_{\|\Phi\|_{C^{0,1}} \leq 1} |\dv \Phi(m)|\bigg)(\Omega) &
\lesssim  \|m\|^{2}_{B^{1/2}_{2,\infty}(\Omega)}\,.
\end{align*}
The first of these estimates is proved in \cite[Proposition~3.10]{GL20},
and the second can be established using similar calculations
which rely on commutator estimates for the function
\begin{align*}
w_\e = 1-|m_\e|^2 =(|m|^2)_\e -|m_\e|^2\,,
\end{align*}
where the subscript $\e$ denotes convolution at scale $\e$.
(In the context of the eikonal equation, arguments based on commutator estimates were introduced in \cite{DeI}.)
The interpolation argument is however a bit involved.
In particular, the constant $C=C(\alpha)$ we are able to obtain in Proposition~\ref{p:ent_prod} blows up as $\alpha\to 0$ or $1$, even though these borderline cases are easier to handle.

The commutator estimates we use in the proof of Proposition~\ref{p:ent_prod} are as follows.

\begin{lem}\label{l:commut}
Let $\Omega\subset\R^2$ an open set and $m\colon\Omega\to\mathbb S^1$.
Let $\rho\in C_c^1(B_1)$, $\rho\geq 0$, $\int\rho=1$, $|\nabla\rho|\leq 2$, and define $m_\e=m*\rho_\e$ for $\rho_\e(x)=\e^{-2}\rho(x/\e)$ and $\e>0$.
Then the commutator 
\begin{align*}
w_\e =1-|m_\e|^2 = (|m|^2 )_\e -|m_\e|^2\,,
\end{align*}
satisfies
\begin{align*}
|w_\e|(x)
&
\lesssim
\Xint{-}_{B_{\e}}|D^h m(x)|^2 
\, dh\,,
\\
|D^{\e k} w_\e|(x)
&
\lesssim |k| 
\Xint{-}_{B_{2\e}}
|D^h m(x)|^2
\, dh\quad \text{for }|k|\leq 1\,,
\end{align*}
for all $x\in\Omega$ such that $B_{2\e}(x)\subset\Omega$.
\end{lem}

Here and in what follows we denote by $D^h$ the finite difference operator
\begin{align*}
D^h m =m^h - m,\quad m^h =m(\cdot +h)\quad\text{for }h\in\R^2\,.
\end{align*}

\begin{proof}[Proof of Lemma~\ref{l:commut}]
As in \cite{DeI}, these commutator estimates come from the commutator identity
\begin{align*}
w_\e(x)
&
= 1-|m_\e(x)|^2
=\int_{B_\e} |m(x- y)-m_\e(x)|^2\, \rho_\e(y)\,dy
\\
&
=\int_{\Omega} \bigg| \int_{\Omega} \big(m(y)-m( z)\big)
\rho_\e(x-z)\,dz
\bigg|^2\rho_\e(x-y)\,dy\,,
\end{align*}
which follows by integrating
\begin{align*}
1-|m_\e(x)|^2
&
=|m(x- y)|^2 -|m_\e(x)|^2
\\
&
=|m(x- y)-m_\e(x)|^2 + 2 \langle m_\e(x),m(x- y)-m_\e(x)\rangle\,,
\end{align*}
with respect to $\rho_\e(y)\, dy$.
The commutator identity directly implies the first estimate.
Applying the finite difference operator,
 it also
 implies the identity
\begin{align*}
D^{\e k}w_\e(x)
&
=
\int_{\Omega} \bigg| \int_{\Omega} \big(m(y)-m( z)\big)
\rho_\e^{\e k}(x-z)\,dz
\bigg|^2D^{\e k}\rho_\e(x-y)\,dy
\\
&
\quad
+ \int_{\Omega}
\Big\langle \int_{\Omega}\big(m(y)-m( z')\big)
(\rho_\e+\rho_\e^{\e k})(x-z')\,dz'
,
\\
&
\hspace{5em}
\int_{\Omega}\big(m(y)-m( z)\big)
D^{\e k}\rho_\e(x-z)\,dz
\Big\rangle
\,\rho_\e(x-y)\, dy\,
\\
&
=
\int_{B_2} \bigg| \int_{B_2} \big(m(x-\e y)-m( x-\e z)\big)
\rho^{k}(z)\,dz
\bigg|^2D^{k}\rho(y)\,dy
\\
&
\quad
+ \int_{B_2}
\Big\langle \int_{\Omega}\big(m(x-\e y)-m( x-\e z')\big)
(\rho +\rho^{ k})(z')\,dz'
,
\\
&
\hspace{5em}
\int_{B_2}\big(m(x-\e y)-m(x-\e z)\big)
D^{k}\rho( z)\,dz
\Big\rangle
\,\rho ( y)\, dy\,,
\end{align*}
which then provides the second estimate.
\end{proof}

We will also use the estimate
\begin{align}\label{eq:estim_grad_meps}
|\nabla m_\e|(x)
&
\lesssim \frac{1}{\e}\Xint{-}_{B_\e} |D^h m(x)|\, dh\,.
\end{align}
which follows from the identity
\begin{align*}
\nabla m_\e(x)
&
=\frac{1}{\e}\int_{B_1}(m(x-\e y)-m(x))\nabla\rho(y)\, dy\,.
\end{align*}

\begin{proof}[Proof of Proposition~\ref{p:ent_prod}]
Let $\Phi\in \mathrm{ENT}$ such that $\|\Phi\|_{C^{1,\alpha}}\leq 1$,
and $\widehat\Phi$ a radial extension $\widehat\Phi(re^{i\theta})=\eta(r)\Phi(e^{i\theta})$ for some $\eta\in C_c^2(0,\infty)$ with $\eta(1)=1$.
Then we have
\begin{align*}
\dv\widehat\Phi(m_\e)
&
=\Psi(m_\e)\cdot\nabla w_\e\,,
\quad w_\e =1-|m_\e|^2\,,
\\
\Psi(re^{i\theta})
&
=
\frac{1}{2r^2}\eta(r)\lambda(e^{i\theta})\, e^{i\theta}
-\frac 1{2r}\eta'(r)\Phi(e^{i\theta})
\,,
\quad \lambda=\partial_\theta\Phi\cdot (ie^{i\theta})\,,
\end{align*}
and
$\|\Psi\|_{C^{0,\alpha}}\lesssim 1$. 

\medskip

Let $U\subset\Omega$ an open subset and $\zeta\in C_c^1(U)$.
We fix an intermediate open set $\Omega'$ and $\delta\in (0,1)$ such that
\begin{align*}
\supp(\zeta) +B_{4\delta}\subset\Omega'
\subset\Omega'+B_{2\delta}\subset U\,,
\end{align*}
and a cut-off function  $\chi\in C_c^1(\Omega)$ such that
\begin{align*}
\mathbf 1_{\supp(\zeta)}\leq \chi\leq 1
\quad
\text{and } \supp(\chi)+B_{2\delta}\subset\Omega'\,.
\end{align*}
Then, for $0<\e<\delta$, we write
\begin{align}\label{eq:dvPhimeps_chi}
&
\langle\dv \widehat\Phi(m_\e),\zeta\rangle
+
\int_{\R^2}w_\e \chi\Psi(m_\e)\cdot\nabla\zeta \, dx
=\int_{\R^2}\chi\Psi(m_\e)\cdot\nabla(\zeta w_\e ) \, dx
\,.
\end{align}
Using the Littlewood-Paley characterization of Besov spaces and Hölder's inequality (see Lemma~\ref{l:dual_besov})  we find 
\begin{align}
\label{eq:estim_LP}
\bigg|
\int_{\R^2}\chi\Psi(m_\e)\cdot\nabla(\zeta w_\e) \, dx
\bigg|
&
\lesssim
\|\chi\Psi(m_\e)\|_{B^{\alpha}_{p,\infty}}  \|\zeta w_\e\|_{B^{1-\alpha}_{q,1}}\,,
\end{align}
where we choose $p=(2+\alpha)/\alpha$, hence $q=(2+\alpha)/2$.
To estimate the two Besov norms in the right-hand side,
 we come back to their finite difference characterization.
We have
\begin{align*}
\|D^h[\chi \Psi(m_\e)]\|_{L^p}
&
=\|(D^h\chi) \Psi(m_\e^h) +\chi D^h [\Psi(m_\e)]\|_{L^p(\Omega' \cap (\Omega'-h))}
\\
&
\quad
+\|\chi \Psi(m_\e)\|_{L^p(\supp(\chi)\setminus (\Omega'-h))}
\\
&
\quad
+\|\chi  \Psi(m_\e) \|_{L^p( \supp(\chi)\setminus(\Omega'+h))}\,.
\end{align*}
For $0<|h|<\delta$, the last two terms are zero and we deduce
\begin{align*}
\frac{1}{|h|^\alpha}
\|D^h[\chi \Psi(m_\e)]\|_{L^p}
&
\lesssim |\Omega|^{\frac 1p}\|\chi\|_{C^1}
\|\Psi\|_\infty
+\frac{1}{|h|^\alpha}\|\chi D^h [\Psi(m_\e)]\|_{L^p(\Omega' \cap (\Omega'-h))}
\\
&
\lesssim |\Omega|^{\frac 1p} \|\chi\|_{C^1}\|\Psi\|_\infty
+\frac{1}{|h|^\alpha}\|\Psi\|_{C^{0,\alpha}}
\| D^h m_\e \|^\alpha_{L^{\alpha p}(\Omega' \cap (\Omega'-h))}
\\
&
\lesssim 
|\Omega|^{\frac 1p} \|\chi\|_{C^1}\|\Psi\|_\infty
+
\|\nabla m_\e \|^\alpha_{L^{\alpha p}(\Omega'+B_\delta)}\,.
\end{align*}
To estimate the last factor we use \eqref{eq:estim_grad_meps} 
which implies
\begin{align*}
\|\nabla m_\e\|_{L^{\alpha p}(\Omega'+B_\delta)}
&
\lesssim
\frac 1\e
\Xint{-}_{B_\e}
\|D^h m\|_{L^{\alpha p}(\Omega'+B_\delta)} \, dh\,.
\end{align*}
Plugging this into the previous estimate for $0<|h|<\delta$
and
recalling that $p=(2+\alpha)/\alpha$
we obtain
\begin{align}
&
\|\chi\Psi(m_\e)\|_{B^{\alpha}_{p,\infty}}
\lesssim
\|\chi\Psi(m_\e)\|_{L^p} +
\sup_{|h|>0} \frac{\|D^h[\chi\Psi(m_\e)]\|_{L^p}}{|h|^\alpha}
\nonumber
\\
&
\lesssim |\Omega|^{\frac {\alpha}{2+\alpha}}\big(\|\chi\|_{C^1} +\delta^{-\alpha} \big)
\|\Psi\|_\infty
+\frac{1}{\e^\alpha}
\Xint{-}_{B_\e}
\|D^h m\|^\alpha_{L^{2+\alpha}(\Omega'+B_\delta)} \, dh\,.
\label{eq:chiPsiBalphapinfty}
\end{align}
Now we turn to estimating the second factor in the right-hand side of \eqref{eq:estim_LP}.
As above, we have
\begin{align*}
\|D^h(\zeta w_\e)\|_{L^q}
&
=\|(D^h\zeta) w_\e^h +\zeta D^h w_\e\|_{L^q(\Omega' \cap (\Omega'-h))}
\\
&
\quad
+\|\zeta w_\e\|_{L^q(\supp(\zeta)\setminus (\Omega'-h))}
+\|\zeta  w_\e \|_{L^q( \supp(\zeta)\setminus(\Omega'+h))}\,.
\end{align*}
Since $0<\e<\delta$, the last two terms are zero if $|h|\leq \e$.
We deduce
\begin{align*}
\|\zeta w_\e\|_{B^{1-\alpha}_{q,1}}
&
\lesssim 
\|\zeta w_\e\|_{L^q}
+
\int_{\R^2} \frac{\|D^h(\zeta w_\e)\|_{L^q}}
{|h|^{1-\alpha}}
\frac{dh}{|h|^2}
\\
&
\lesssim 
 \|\zeta\|_{C^1} \|w_\e\|_{L^q(\Omega')}
+\|\zeta\|_\infty \int_{|h|\geq\e} \frac{\| w_\e\|_{L^q(\supp(\zeta))}}{|h|^{1-\alpha}}\frac{dh}{|h|^2}
\\
&
\quad
+\|\zeta\|_\infty \int_{|h|\leq\e} \frac{\|D^hw_\e\|_{L^q(\supp(\zeta))}}{|h|^{1-\alpha}}\frac{dh}{|h|^2}
\\
&
\lesssim 
\|\zeta\|_{C^1} \|w_\e\|_{L^q(\Omega')}
+\frac{\|\zeta\|_\infty}{1-\alpha}  \frac{\| w_\e\|_{L^q(\supp(\zeta))}}{\e^{1-\alpha}}
\\
&
\quad
+\frac{ \|\zeta\|_\infty}{\e^{1-\alpha}} \int_{|k|\leq 1} \frac{\|D^{\e k } w_\e\|_{L^q(\supp(\zeta))}}{|k|^{1-\alpha}}\frac{dk}{|k|^2}\,.
\end{align*}
Recalling that $q=(2+\alpha)/2$ and
using the commutator estimates of Lemma~\ref{l:commut} we 
infer
\begin{align*}
\|\zeta w_\e\|_{B^{1-\alpha}_{q,1}}
&
\lesssim \|\zeta\|_{C^1}
\sup_{|h|\leq\e}
 \|D^h m\|^2_{L^{2+\alpha}(\Omega')}
 \\
 &
 \quad
+\frac{\|\zeta\|_\infty}{1-\alpha}  \frac{1}{\e^{1-\alpha}}
\Xint{-}_{B_\e}
 \|D^h m\|^2_{L^{2+\alpha}(\Omega')}\, dh
\\
&
\quad
+\frac{ \|\zeta\|_\infty}{\e^{1-\alpha}}\frac{1}{\alpha} 
\Xint{-}_{B_{2\e}}
 \|D^h m\|^2_{L^{2+\alpha}(\Omega')}
 dh
 \\
 &
 \lesssim \e^{\frac{2}{2+\alpha}} \|\zeta\|_{C^1} |m|^2_{B^{1/(2+\alpha)}_{2+\alpha,\infty}}
 \\
 &
 \quad
 +
 \frac{\|\zeta\|_\infty}{\alpha(1-\alpha)}
\frac{1}{\e^{1-\alpha}}
  \Xint{-}_{B_{2\e}}
 \|D^h m\|^2_{L^{2+\alpha}(\Omega')} dh\,.
\end{align*}
Using this and \eqref{eq:chiPsiBalphapinfty} 
to estimate the right-hand side of \eqref{eq:estim_LP},
 noting that the second term in the left-hand side of \eqref{eq:dvPhimeps_chi} converges to 0 as $\e\to 0$ 
(since $w_\e\to 0$ in $L^1$),
and also that 
\begin{align*}
\frac{1}{\e^{1-\alpha}}\sup_{|h|<2\e}\|D^h m\|^2_{L^{2+\alpha}} \lesssim \e^{\frac{2}{2+\alpha}+\alpha -1}|m|^2_{B^{1/(2+\alpha)}_{2+\alpha,\infty}}\to 0\,,
\end{align*}
 we deduce
\begin{align*}
&
\big|
\langle \dv\widehat\Phi(m_\e),\zeta\rangle
\big| -o(1)
\\
&
\lesssim
\frac{\|\zeta\|_\infty}{\alpha(1-\alpha)}
\frac{1}{\e}
  \Xint{-}_{B_{2\e}}
 \|D^h m\|^2_{L^{2+\alpha}(U)} dh
 \,
   \Xint{-}_{B_{\e}}
 \|D^k m\|^\alpha_{L^{2+\alpha}(U)} dk\,,
\end{align*}
where $o(1)\to 0$ as $\e\to 0$.
Using Young's inequality $ab\leq a^p/p +b^q/q$ for $a,b\geq 0$, 
as well as Jensen's inequality,
this implies
\begin{align*}
&
\big|
\langle \dv\widehat\Phi(m_\e),\zeta\rangle
\big| 
\lesssim
\frac{\|\zeta\|_\infty}{\alpha(1-\alpha)}
\frac{1}{\e}
  \Xint{-}_{B_{2\e}}
\int_{U} |D^h m(x)|^{2+\alpha}dx\, dh
+o(1)\,.
\end{align*}
This is valid for any $\Phi\in\mathrm{ENT}$ with $\|\Phi\|_{C^{1,\alpha}}\leq 1$, any open $U\subset\Omega$ and any $\zeta\in C_c^1(U)$. 
Letting $\e\to 0$ and taking the supremum over functions $\zeta$ with $\|\zeta\|_\infty\leq 1$, we deduce
\begin{align*}
|\dv\Phi(m)|(U)\lesssim
\frac{1}{\alpha(1-\alpha)}
\liminf_{\e\to 0}
\frac{1}{\e}
  \Xint{-}_{B_{2\e}}
\int_{U} |D^h m(x)|^{2+\alpha}dx\, dh\,.
\end{align*}
Applying this to any finite collection $\Phi_1,\ldots,\Phi_N\in\mathrm{ENT}$ with $\|\Phi_j\|_{C^{1,\alpha}}\leq 1$,
any disjoint collection of open subsets 
$U_1,\ldots,U_N\subset V\subset\subset\Omega$, 
we deduce that
\begin{align*}
\sum_{j=1}^N |\dv\Phi(m)|(U_j) 
&
\lesssim \frac{1}{\alpha(1-\alpha)}
\liminf_{\e\to 0}
\frac{1}{\e}
  \Xint{-}_{B_{2\e}}
\int_{V} |D^h m(x)|^{2+\alpha}dx\, dh
\\
&
\lesssim \frac{1}{\alpha(1-\alpha)}\|m\|^{2+\alpha}_{B^{1/(2+\alpha)}_{2+\alpha,\infty}(\Omega)}\,.
\end{align*}
The desired estimate on the supremum measure follows thanks to the inner regularity of Radon measures.
\end{proof}

\begin{lem}\label{l:dual_besov}
For any $p,p'\in [1,\infty]$ such that $1/p+1/p'=1$
and any $\alpha\in (0,1)$ we have
\begin{align*}
\bigg|
\int_{\R^2}f\nabla g\, dx 
\bigg|
\lesssim
\|f\|_{B^{\alpha}_{p,\infty}}\|g\|_{B^{1-\alpha}_{p',1}}\,,
\end{align*}
for all $(f,g)\in B^{\alpha}_{p,\infty}(\R^2)\times B^{1-\alpha}_{p',1}(\R^2)$.
\end{lem}
\begin{proof}[Proof of Lemma~\ref{l:dual_besov}]
This inequality essentially amounts to the inclusion of $B^{-\alpha}_{p',1}$ into the dual of $B^{\alpha}_{p,\infty}$.
The elementary proof  is basically contained in \cite[\S~2.11]{triebel83}, 
where the dual of $B^{\alpha}_{p,q}$ is shown to be equal to $B^{-\alpha}_{p',q'}$ for all $p,q\in (1,\infty)$.
Here we have $q=\infty$ and only one inclusion is true, 
that is why this statement is not stated explicitly there.
We provide the proof for the readers' convenience.

It relies on the Littlewood-Paley characterization of Besov spaces, which we start by recalling.
We fix a smooth partition of unity $\lbrace\chi_j\rbrace_{j\geq 0}\subset C_c^\infty(\R^2)$ with the properties that
\begin{align*}
&
\sum_{j=0}^\infty \chi_j(x)=1,\\
&  |\chi_0(\xi)|\leq \mathbf 1_{|\xi|\leq 2}\,,
\\
&
|\chi_j(\xi)|\leq\mathbf 1_{2^{-j-1}\leq |\xi|\leq 2^{j+1}}
\text{ for }j\geq 1\,,
\\
& \sup_{j\geq 0} 2^{jk}\sup_{\R^2}|\nabla^k\chi_j| <\infty\quad\forall k\geq 0\,.
\end{align*}
Then, for any $\gamma\in\R$ and $p,q\in [1,\infty]$, the Besov space $B^{\gamma}_{p,q}(\R^2)$ consists of all tempered distributions $\varphi \in\mathcal S'(\R^2)$ such that the norm
\begin{align*}
\|\varphi \|_{B^{\gamma}_{p,q}} =\big\| \big(2^{j\gamma}\|\mathcal F^{-1}\chi_j\mathcal F \varphi\|_{L^p}\big)_{j\geq 0}\big\|_{\ell^q}\,,
\end{align*}
is finite  \cite[\S~2.3.1]{triebel83},
where $\mathcal F$ denotes the Fourier transform on $\mathcal S'(\R^2)$.
Moreover, for $\gamma\in (0,1)$, these norms (which depend on the system $\lbrace\chi_j\rbrace$) are equivalent to 
\begin{align*}
\|\varphi \|_{B^{\gamma}_{p,q}} =
\|\varphi\|_{L^p} + \big\| |h|^{-\gamma} \|D^h\varphi\|_{L^p}\big\|_{L^q(dh/h^2)}
\,,
\end{align*}
see e.g. \cite[\S{2.5.12}]{triebel83}.

To prove the claimed inequality, we use the decomposition
\begin{align*}
\varphi =\sum_{j\geq 0}\mathcal F^{-1}\chi_j\mathcal F \varphi\,,
\end{align*}
and the fact that $\chi_j\chi_k\equiv 0$ for $|j-k|\geq 2$,
to rewrite the integral as
\begin{align*}
\int_{\R^2}f\nabla g\, dx
&
=\langle f,\nabla g\rangle
=\sum_{j,k\geq 0}
\langle \mathcal F^{-1}\chi_j\mathcal F f,\mathcal F^{-1}\chi_k i\xi \mathcal F g\rangle
\\
&
=\sum_{j,k\geq 0}
\langle \chi_j\mathcal F f,\chi_k i\xi \mathcal F g\rangle
\\
&
=\sum_{r=-1}^1 \sum_{j\geq 0}
\langle \chi_{j+r}\mathcal F f,\chi_j i\xi\mathcal F g\rangle
\\
&
=\sum_{r=-1}^1 \sum_{j\geq 0}
\langle\mathcal F^{-1} \chi_{j+r}\mathcal F f,\mathcal F^{-1}\chi_j i\xi\mathcal F g\rangle\,.
\end{align*}
Recalling that $\chi_j$ is supported in $2^{j-1}\leq |\xi|\leq 2^{j+1}$ for $j\geq 2$, 
and applying Hölder's inequality, we infer
\begin{align*}
\bigg|\int_{R^2} f\nabla g\, dx\bigg|
&
\leq \sum_{r=-1}^1
\sum_{j\geq 0}
\|\mathcal F^{-1} \chi_{j+r}\mathcal F f\|_{L^p}
\|\mathcal F^{-1}\chi_j i\xi \mathcal F g\|_{L^{p'}}
\\
&
\lesssim
\sum_{r=-1}^1
\sum_{j\geq 0}
2^j\|\mathcal F^{-1} \chi_{j+r}\mathcal F f\|_{L^p}
\|\mathcal F^{-1}\chi_j  \mathcal F g\|_{L^{p'}}\,.
\end{align*}
The last inequality follows
 from the properties of $\lbrace\chi_j\rbrace$ and 
a Fourier multiplier theorem (see e.g. \cite[\S{1.5}]{triebel83}).
Writing $2^j=2^{\alpha j} 2^{(1-\alpha)j}$, we deduce
\begin{align*}
\bigg|\int_{R^2} f\nabla g\, dx\bigg|
&
\lesssim
\sup_{k\geq 0}
2^{\alpha k}\|\mathcal F^{-1} \chi_{k}\mathcal F f\|_{L^p}
\sum_{j\geq 0}
2^{(1-\alpha)j}\|\mathcal F^{-1}\chi_j  \mathcal F g\|_{L^{p'}}\,,
\end{align*}
which corresponds to the claimed inequality.
\end{proof}

\section{Kinetic formulation and Lagrangian representation}

In this section we recall the notions of kinetic formulation and of Lagrangian representation introduced in \cite{jabinperthame01,GL20} and \cite{marconi21ellipse} respectively. 
We prove moreover some properties which relate the traces on the jump set $J_m$ to the traces of Lagrangian trajectories.

\begin{thm}\label{t:kin}
Let $m \in B^{\frac13}_{3,\infty}(\Omega)$ be a solution of \eqref{eq:eik}. Then there is $\sigma \in \mathcal M(\Omega \times \T)$ such that 
\begin{equation}\label{eq:kin}
e^{is}\cdot \nabla_x \chi = \partial_s \sigma,
\end{equation}
where 
\[
\chi(x,s)= 
\begin{cases}
1 & \mbox{if }e^{is}\cdot m(x)>0 , \\
0 & \mbox{otherwise}.
\end{cases}
\]
Among the measures $\sigma$ satisfying \eqref{eq:kin}, there is a unique $\sigma_{\min}$ minimizing $\|\sigma\|$ and it has the following structure: 
\[
\sigma_{\min} = \nu_{\min} \otimes (\sigma_{\min})_x, \qquad \mbox{where} \quad \nu_{\min}= (p_x)_\sharp |\sigma_{\min}|
\]
and for $\nu_{\min}$-a.e. $x \in \Omega \setminus J_m$ it holds
\begin{equation}\label{e:dis-cont}
(\sigma_{\min})_x = \pm \frac12 \left( \delta_{\mathfrak s(x)}+ \delta_{\mathfrak s(x)+ \pi}\right)
\end{equation}
for some $\mathfrak s: \Omega \to \T$ uniquely defined $\nu_{\min}\llcorner J_m^c$-a.e., and for $\nu_{\min}$-a.e. $x \in J_m$ it holds
\[
(\sigma_{\min})_x = \mathbf{n}\cdot e^{i\bar s} \bar g_\beta(s-\bar s)\L^1,
\]
where ${\bf n}(x)$ is the normal to $J_m$ at $x$, the values $ \bar s\in \T$, $\beta \in (0,\frac{\pi}2)$ are uniquely determined by 
\[
m^+(x)= e^{i(\bar s + \beta)}, \quad m^-(x)= e^{i(\bar s - \beta)}
\]
and $\bar g_\beta$ is a Lipschitz function.
If moreover $\beta \in (0,\frac\pi4)$, then $\bar g_\beta$ is supported on $[ - \frac\pi2 - \beta, - \frac\pi2 + \beta] \cup [ \frac\pi2 - \beta, \frac\pi2 + \beta]$ and is non-negative.
\end{thm}
We refer to \cite{GL20} for the kinetic formulation \eqref{eq:kin} and to \cite{marconi21ellipse} for the structure of the kinetic measure $\sigma_{\min}$, where the explicit expression of $\bar g_\beta$ is computed (see also \cite{LP23facto}).
We observe also that $\bar g_\beta$ approaches $\frac12\delta_{\pi/2}+ \frac12\delta_{-\pi/2}$ as $\beta \to 0$ matching \eqref{e:dis-cont}.

For future use, we recall that the kinetic measure $\sigma_{\min}$ encodes the dissipation of the entropies for a large class of entropies: to any $\psi\in C^1(\T;\R)$ we can associate the entropy
\begin{align}\label{e:Phipsi}
\Phi_\psi(z)=\int_{\T} \mathbf 1_{e^{is}\cdot z >0} \psi(s)e^{is}\, ds\,.
\end{align}
An entropy $\Phi$ belongs to the class $\widetilde{\mathrm{ENT}}$ defined in \eqref{eq:tildeENT}
 if and only if $\Phi=\Phi_\psi$ for some $\pi$-periodic $\psi \in C^1(\T;\R)$, see \cite{marconi21ellipse}.
Integrating \eqref{eq:kin} in the $s$-variable tested with $\psi(s)$, we obtain
\begin{equation}\label{e:ent-kin}
 \dv \Phi_\psi(m) = \left( - \int_\T \psi'(s) d (\sigma_{\min})_x(s) \right) \nu_{\min}
\end{equation}
We now recall the notion of Lagrangian representation.
Given $T>0$ we let
\begin{align*}
\Gamma= \Big\{ (\gamma,t^-_\gamma,t^+_\gamma)\colon
& 0\le t^-_\gamma\le t^+_\gamma\le T, \\
&
\gamma=(\gamma_x,\gamma_s)\in \BV((t^-_\gamma,t^+_\gamma);\Omega \times \R/2\pi \Z),
 \gamma_x \mbox{ is Lipschitz} \Big\}.
\end{align*}
We will always consider the right-continuous representative of the component $\gamma_s$ and we will write $\gamma(t^-_\gamma)$ instead of $ \lim_{t\to t^-_\gamma} \gamma(t)$ and  $\gamma(t^+_\gamma)$ instead of $ \lim_{t\to t^+_\gamma} \gamma(t)$.
For every $t \in (0,T)$ we consider the section 
\begin{equation*}
 \Gamma(t):= \left\{\left(\gamma,t^-_\gamma,t^+_\gamma\right)\in  \Gamma: t \in \left(t^-_\gamma,t^+_\gamma\right)\right\}
\end{equation*}
and we denote by 
\begin{align*}
\begin{array}{crl}
 e_t\colon &
  \Gamma(t) &\longrightarrow \Omega \times \R/2\pi \Z \\
&(\gamma, t^-_\gamma, t^+_\gamma)  &\longmapsto  \gamma(t),
\end{array}
\end{align*}
the evaluation map at time $t$.

\begin{definition}\label{D_Lagr}
Let $\Omega$ be a $C^{1,1}$ open set and $m$ solving \eqref{eq:eik} and \eqref{eq:kin}.
We say that a finite non-negative Radon measure $\omega \in \mathcal M( \Gamma)$ is a \emph{Lagrangian representation} of $m$ if the following conditions are satisfied:
\begin{enumerate}
\item for every $t\in (0,T)$ we have
\begin{equation}\label{E_repr_formula}
( e_t)_\sharp \left[ \omega \llcorner  \Gamma(t)\right]= \mathbf 1_{E_m}\, \L^{2}\times \L^1,
\end{equation}
where $E_m\subset\Omega\times\R/2\pi\Z$ is the `epigraph'
\begin{align*}
E_m=\left\lbrace (x,s)\in\Omega\times\R/2\pi\Z\colon m(x)\cdot e^{is}> 0\right\rbrace;
\end{align*}
\item the measure $\omega$ is concentrated on curves $(\gamma,t^-_\gamma,t^+_\gamma)\in  \Gamma$ solving the characteristic equation:
\begin{equation}\label{e:charac}
\dot\gamma_x(t)= e^{i \gamma_s(t)}\qquad\text{for a.e. }t\in (t^-_\gamma,t^+_\gamma);
\end{equation}
\item we have the integral bound
\begin{equation*}\label{E_reg}
\int_{ \Gamma} \TV_{(0,T)} \gamma_s d\omega(\gamma) <\infty;
\end{equation*}
\item for $\omega$-a.e. $(\gamma,t^-_\gamma,t^+_\gamma)\in \Gamma$ we have
\begin{equation}\label{e:extreme_time}
t^-_\gamma>0 \Rightarrow  \gamma_x(t^-_\gamma ) \in \partial \Omega, \qquad \mbox{and} \qquad 
t^+_\gamma<T \Rightarrow  \gamma_x(t^+_\gamma) \in \partial \Omega.
\end{equation}
\end{enumerate}
Moreover, we say that a Lagrangian representation $\omega$ of $m$ is minimal if
\[
\int_{ \Gamma} \TV_{(0,T)} \gamma_s d\omega(\gamma) = T |\sigma_{\min}|(\Omega),
\]
where $\sigma_{\min}$ is defined in Theorem \ref{t:kin}.
\end{definition}

Given $v \in BV(I;\R^n)$ for some interval $I\subset \R$ we consider the decomposition of the derivative
\[
Dv = \tilde D v + D^j v,
\]
where $\tilde D v$ is the sum of the absolutely continuous and Cantor part of the measure $Dv$ and $D^jv$ is the jump part of $Dv$, see for example \cite{ambrosio}.

\begin{thm}
[\cite{marconi21micromag,marconi21ellipse}]
\label{t:lagr}
Let $\Omega$ be a $C^{1,1}$ open set and $m \in B^{1/3}_{3,\infty}(\Omega)$ solving \eqref{eq:eik} and \eqref{eq:kin}. Then there is a minimal Lagrangian representation $\omega$ of $m$. Moreover the following equalities between measures hold:
\begin{equation}\label{e:kin-lagr}
\begin{split}
\L^1\llcorner [0,T] \times \sigma_{\min} = &~ - \int_\Gamma \sigma_\gamma d\omega(\gamma), \\
\L^1\llcorner [0,T] \times |\sigma_{\min}| = &~ \int_\Gamma |\sigma_\gamma| d\omega(\gamma),
\end{split}
\end{equation}
where 
\begin{equation}\label{e:def-sigma-gamma}
\sigma_\gamma = (\mathrm{Id}, \gamma)_\sharp  \tilde D_t \gamma_s + \mathcal H^1\llcorner E_{\gamma}^+ -\mathcal H^1\llcorner E_{\gamma}^-
\end{equation}
and
\begin{equation}\label{E_def_Epmgamma}
\begin{split}
E_\gamma^+ &:= 
\{(t,x,s) \in (0,T)\times \Omega\times \T: 
\\
&\hspace{2.1em}
\gamma_x(t)=x 
 \mbox{ and } \gamma_s(t-)\le s \le \gamma_s(t+)\le \gamma_s(t-)+\pi \}, \\
E_\gamma^- & := \{(t,x,s) \in (0,T)\times \Omega\times \T:
\\
&\hspace{2.1em} \gamma_x(t)=x  \mbox{ and } \gamma_s(t+)\le s \le \gamma_s(t-)< \gamma_s(t+)+\pi \}.
\end{split}
\end{equation}
Accordingly, for any $\psi \in C^1(\T)$ we can disintegrate the entropy production of $\Phi_\psi$ defined in \eqref{e:Phipsi} along the Lagrangian curves:
\begin{align}\label{e:repr_entr}
\langle \dv\Phi_\psi(m),\zeta\rangle
&
=-\frac1T\int_{\Gamma}  \int_{I_\gamma}\zeta(\gamma_x(t))D(\psi\circ\gamma_s)(dt)\, d\omega(\gamma)\,,
\end{align}
for any $\zeta\in C^1(\Omega)$.
\end{thm}

In general it is not true that for any $\zeta\in C^1(\Omega)$ and any $\psi \in C^1(\T)$ it holds
\begin{equation}\label{e:repr_entr_abs}
\langle |\dv\Phi_\psi(m)|,\zeta\rangle
=\frac1T\int_{\Gamma}  \int_{I_\gamma}\zeta(\gamma_x(t))|D(\psi\circ\gamma_s)|(dt)\, d\omega(\gamma).
\end{equation}
Before discussing the validity of a weaker version of the above formula, we prove a result stating that the decomposition in jump part and continuous part of $\nu_{\min}$ is compatible with the corresponding decomposition of Lagrangian trajectories.
\begin{lem}\label{l:jumpLE}
Let $\omega$ be a minimal Lagrangian representation of $m \in B^{1/3}_{3,\infty}(\Omega)$ solving \eqref{eq:eik}.
Then 
\[
\begin{split}
&\nu_{\min}\llcorner J_m =  \frac1T \int_\Gamma (\gamma_x)_\sharp |D^j\gamma_s| d\omega (\gamma), \\
&\nu_{\min}\llcorner (\Omega \setminus J_m) =  \frac1T \int_\Gamma (\gamma_x)_\sharp |\tilde D\gamma_s| d\omega (\gamma).
\end{split}
\]
\end{lem}
\begin{proof}
It is sufficient to prove that for $\omega$-a.e. $\gamma$ the following holds:
\[
|D^j\gamma_s|(\{\gamma_x \in \Omega \setminus J_m\})=0, \qquad |\tilde D\gamma_s|(\{ \gamma_x \in J_m\})=0.
\]
The first equality follows from
\cite[Lemma~3.5]{marconi21ellipse}
which states that for $\omega$-a.e. $\gamma \in \Gamma$ the following holds: for every $t \in I_\gamma$ such that $\gamma_s(t+)\ne \gamma_s(t-)$, we have
$\gamma_x(t) \in J_m$.
%
The second equality follows from the fact that for $\omega$-a.e. $\gamma\in \Gamma$ the set $\{\gamma_x \in J_m\}$ is at most countable,
since $J_m$ is 
countably 1-rectifiable and $\lbrace\gamma_x\in \Sigma\rbrace$ is finite for any Lipschitz curve $\Sigma$, 
see the proof of \cite[Lemma~3.4]{CHLM22} which adapts directly to our setting.
\end{proof}

We now prove that \eqref{e:repr_entr_abs} holds outside the jump set.

\begin{lem}\label{l:abs_no_jumps}
Let $m \in B^{1/3}_{3,\infty}(\Omega)$ solving \eqref{eq:eik} and \eqref{eq:kin} and $\omega$ be a minimal Lagrangian representation of $m$.
Then for every $\pi$-periodic $\psi \in C^1(\T)$  and any $A \subset \Omega \setminus J_m$ it holds
\[
\begin{split}
[\dv \Phi_\psi (m)]_+ (A) & = \frac1T \int_\Gamma [D(\psi \circ \gamma_s)]_- (\gamma_x \in A) \, d\omega (\gamma),\\
[\dv \Phi_\psi (m)]_- (A) & = \frac1T \int_\Gamma [D(\psi \circ \gamma_s)]_+ (\gamma_x \in A) \, d\omega (\gamma).
\end{split}
\]
In particular
\[
|\dv \Phi_\psi (m)| (A) = \frac1T \int_\Gamma |D(\psi \circ \gamma_s)| (\gamma_x \in A) d\omega (\gamma) \qquad \mbox{for any } A \subset \Omega \setminus J_m.
\]
\end{lem}
\begin{proof}
Thanks to Theorem \ref{t:kin}, the measure $(\sigma_{\min})_x$ has a definite sign on $\T$ for $\nu_{\min}$-a.e. $x \in J_m^c$.
More precisely, we may fix two disjoint Borel sets 
$\tilde\Omega^\pm\subset\Omega\setminus J_m$
such that 
\begin{align*}
&\nu_{\min}\big(\Omega\setminus (J_m\cup\tilde\Omega^+\cup\tilde\Omega^-)\big)=0\,,\\
&
(\sigma_{\min})_x  
=  \frac 12 \left( \delta_{\mathfrak s(x)}+ \delta_{\mathfrak s(x)+ \pi}\right)
\geq 0 \text{ on }\T\, \text{ for all }x\in\tilde\Omega^+\,,\\
&
(\sigma_{\min})_x 
=
-  \frac 12 \left( \delta_{\mathfrak s(x)}+ \delta_{\mathfrak s(x)+ \pi}\right)
\leq 0 \text{ on }\T\, \text{ for all }x\in\tilde\Omega^-\,.
\end{align*}
Combining this with \eqref{e:ent-kin} and the fact that $\psi$ is $\pi$-periodic, we see that
\begin{align*}
\dv\Phi_\psi(m)\llcorner (\Omega\setminus J_m)
=\big(\mathbf 1_{\tilde\Omega^-}-\mathbf 1_{\tilde\Omega^+}\big) \big(\psi'\circ \mathfrak s \big) \, \nu_{\min}\llcorner (\Omega\setminus J_m)\,,
\end{align*}
hence the positive, resp. negative, part of the measure
 $\dv\Phi_\psi(m)$ outside $J_m$ is concentrated on the set
 $\tilde A^+$, resp. $\tilde A^-$, where
 \begin{align*}
\tilde A^+ 
&
= \big\lbrace (\mathbf 1_{\tilde\Omega^-}-\mathbf 1_{\tilde\Omega^+}\big) (\psi'\circ \mathfrak s )>0\big\rbrace\,,
\\
\tilde A^- 
&
= \big\lbrace (\mathbf 1_{\tilde\Omega^-}-\mathbf 1_{\tilde\Omega^+}\big) (\psi'\circ \mathfrak s ) < 0\big\rbrace
\,.
 \end{align*} 
Moreover, thanks to \eqref{e:repr_entr} and Lemma~\ref{l:jumpLE}, for any Borel set $A\subset\Omega\setminus J_m$, we have
\begin{align*}
\big(\dv\Phi_\psi(m)\big) (A)
&=-\frac 1T \int_\Gamma \tilde D(\psi\circ \gamma_s)(\lbrace \gamma_x\in A\rbrace)\, d\omega(\gamma)\,.
\end{align*}
Further note that, thanks to  \eqref{e:kin-lagr}-\eqref{e:def-sigma-gamma} and the expression of 
$(\sigma_{\min} )_x$ for $x\notin J_m$,
for $\omega$-a.e. $\gamma\in\Gamma$
 we have
$\gamma_s(t)\in\lbrace \mathfrak s(\gamma_x(t)),\mathfrak s(\gamma_x(t))+\pi\rbrace$  for $\tilde D\gamma_s$-a.e. $t\in I_\gamma$, so the above integrand can be rewritten as
\begin{align*}
\tilde D(\psi\circ \gamma_s)(\lbrace \gamma_x\in A\rbrace)
= 
\int_{I_\gamma} (\psi'\circ \mathfrak s) (\gamma_x(t))\,
\mathbf 1_{\gamma_x(t)\in A} 
\tilde D\gamma_s(dt)\,.
\end{align*}
Moreover, thanks again to \eqref{e:kin-lagr}-\eqref{e:def-sigma-gamma} and Lemma~\ref{l:jumpLE} we have
\begin{align*}
(\tilde D\gamma_s)_+(\{\gamma_x \in \tilde\Omega^+\})=(\tilde D\gamma_s)_-(\{ \gamma_x \in \tilde\Omega^-\})=0\quad\text{for }\omega\text{-a.e. }\gamma\in\Gamma\,.
\end{align*}
From the two last equations and the definitions of $\tilde A^\pm$ we deduce, for $\omega$-a.e. $\gamma\in\Gamma$, that
\begin{align*}
\tilde D(\psi\circ \gamma_s)(\lbrace \gamma_x\in A\rbrace)
\begin{cases}
\geq 0\quad
&\text{if }A\subset \tilde A^-\,,
\\
\leq 0\quad
&\text{if }A\subset \tilde A^+\,,
\end{cases}
\end{align*}
and this implies the claimed decomposition.
\end{proof}

The proof of the above lemma relies on the following properties: for $\nu$-a.e. $x \in \Omega \setminus J_m$  the measure $(\sigma_{\min})_x$ has a definite sign and $\psi'$ does not change sign on $\mathrm{supp}\, (\sigma_{\min})_x$.
Hence the analysis extends to the case of small shocks for particular entropies:
given $\delta>0$ denote by
\begin{equation}\label{e:Jmdelta}
J_m^\delta = \{x \in J_m : m^\pm(x)=e^{is^\pm(x)}, \ |s^+(x)-s^-(x)|> \delta \}.
\end{equation}
\begin{lem}\label{l:pos-neg-jumps}
Let $m \in B^{1/3}_{3,\infty}(\Omega)$ solving \eqref{eq:eik} and \eqref{eq:kin} and $\omega$ be a minimal Lagrangian representation of $m$.
Let $\psi \in C^{0,1}(\T)$ be $\pi$-periodic and let $\Phi_\psi \in \widetilde{\mathrm{ENT}}$ be as in \eqref{e:Phipsi}. Let $\delta \in \left(0, \frac\pi2\right)$ and assume that for every $\tilde s \in \T$ one of the two conditions hold:
\[
\psi'(s)\ge 0 \quad \mbox{for a.e. } s \in [\tilde s, \tilde s + \delta], \qquad \mbox{or} \qquad \psi'(s)\ge 0 \quad  \mbox{for a.e. }  s \in [\tilde s, \tilde s + \delta].
\]
Then for any  $A \subset J_m \setminus J_m^\delta$ it holds
\[
\begin{split}
[\dv \Phi_\psi (m)]_+ (A) & = \frac1T \int_\Gamma [D(\psi \circ \gamma_s)]_- (\gamma_x \in A) \, d\omega (\gamma),\\
[\dv \Phi_\psi (m)]_- (A) & = \frac1T \int_\Gamma [D(\psi \circ \gamma_s)]_+ (\gamma_x \in A) \, d\omega (\gamma).
\end{split}
\]
In particular, we have
\[
|\dv \Phi_\psi (m)| (A) = \frac1T \int_\Gamma |D(\psi \circ \gamma_s)| (\gamma_x \in A) d\omega (\gamma) \quad \mbox{for any } A \subset J_m \setminus J_m^\delta.
\]
\end{lem}
\begin{proof}
With the same notation as in Theorem \ref{t:kin}, let
\[
\tilde J^+ = \{x \in J_m \setminus J_m^\delta : \mathbf n \cdot e^{i\bar s}>0\}, \quad \tilde J^- = \{x \in J_m \setminus J_m^\delta : \mathbf n \cdot e^{i\bar s}<0\}.
\]
For each $x \in J_m\setminus J_m^\delta$ the corresponding half-amplitude $\beta \in (0, \frac\delta2) \subset (0,\frac\pi4)$. Since in this range of $\beta$ we have $g_\beta\ge 0$, then 
\begin{align*}
&\nu_{\min}\big(J_m \setminus (J_m^\delta\cup\tilde J^+\cup\tilde J^-)\big)=0\,,\\
&
(\sigma_{\min})_x  
= \mathbf{n} \cdot e^{i\bar s} g_\beta (\cdot -\bar s) \mathcal L^1
\geq 0 \text{ on }\T\, \text{ for all }x\in\tilde J^+\,,\\
&
(\sigma_{\min})_x 
=
\mathbf{n} \cdot e^{i\bar s} g_\beta (\cdot -\bar s) \mathcal L^1\leq 0 \text{ on }\T\, \text{ for all }x\in\tilde J^-\,.
\end{align*}
By \eqref{e:ent-kin} it follows that 
\[
\dv \Phi_\psi (m) \llcorner (J_m\setminus J_m^\delta) = \left( \int_\T \big(\mathbf 1_{\tilde\Omega^-}-\mathbf 1_{\tilde\Omega^+}\big) \psi'(s) d |(\sigma_{\min})_x| (s)\right) \nu_{\min}\llcorner (J_m\setminus J_m^\delta).
\]
Moreover, by the assumption on $\psi$, we have that $\psi'$ has constant sign on 
$\supp (\sigma_{\min})_x \subset [\bar s -\frac\delta 2, \bar s + \frac \delta 2] \cup [\bar s -\frac\delta 2 + \pi, \bar s + \frac \delta 2 + \pi]$. Therefore the positive and the negative parts of $\dv \Phi_\psi(m)$ restricted to $J_m\setminus J_m^\delta$ are concentrated respectively on the sets
\[
\begin{split}
 A^{j,+} =  &~ \left\{x \in J_m \setminus J_m^\delta: \, \exists s \in \supp  (\sigma_{\min})_x \, ,\,  \big(\mathbf 1_{\tilde J^-}(x)-\mathbf 1_{\tilde J^+}(x)\big) \psi'(s)>0 \right\}, 
\\
 A^{j,-} = &~ \left\{x \in J_m \setminus J_m^\delta: \, \exists s \in \supp  (\sigma_{\min})_x \, ,\,  \big(\mathbf 1_{\tilde J^-}(x)-\mathbf 1_{\tilde J^+}(x)\big)\psi'(s)>0 \right\}.
\end{split}
\]
Thanks to \eqref{e:repr_entr} and Lemma~\ref{l:jumpLE}, for any Borel set $A\subset J_m\setminus J_m^\delta$, we have
\begin{align}\label{e:jump-part}
\big(\dv\Phi_\psi(m)\big) (A)
&=-\frac 1T \int_\Gamma  D^j(\psi\circ \gamma_s)(\lbrace \gamma_x\in A\rbrace)\, d\omega(\gamma)\, \\
&=-\frac 1T \int_\Gamma \sum_{t \in J_\gamma}\left( \psi(\gamma_s(t+))-\psi(\gamma_s(t-)) \right) \mathbf 1_{\{\gamma_x(t)\in A\}} \, d\omega(\gamma).
\nonumber
\end{align}
Further note that, thanks to  \eqref{e:kin-lagr}-\eqref{e:def-sigma-gamma} and the properties of 
$(\sigma_{\min} )_x$ for $x\in J_m\setminus J_m^\delta$,
for $\omega$-a.e. $\gamma\in\Gamma$
 we have
$\gamma_s(t+),\gamma_s(t-)\in [ \bar s (\gamma_x(t))-\frac\delta2,  \bar s (\gamma_x(t))+\frac\delta2]$  or $\gamma_s(t+),\gamma_s(t-)\in [ \bar s (\gamma_x(t))-\frac\delta2+\pi ,  \bar s (\gamma_x(t))+\frac\delta2 + \pi]$ for $ D^j\gamma_s$-a.e. $t\in I_\gamma$.
Moreover, thanks again to \eqref{e:kin-lagr}-\eqref{e:def-sigma-gamma} and Lemma~\ref{l:jumpLE} we have
\begin{align*}
( D^j\gamma_s)_+(\{\gamma_x\in \tilde J^+\})=( D^j\gamma_s)_-(\{\gamma_x \in \tilde J^-\})=0\quad\text{for }\omega\text{-a.e. }\gamma\in\Gamma\,.
\end{align*}
From the above property, the definitions of $A^{j,\pm}$ and \eqref{e:jump-part}we deduce, for $\omega$-a.e. $\gamma\in\Gamma$, that
\begin{align*}
 D^j(\psi\circ \gamma_s)(\lbrace \gamma_x\in A\rbrace)
\begin{cases}
\geq 0\quad
&\text{if }A\subset  A^{j,-}\,,
\\
\leq 0\quad
&\text{if }A\subset  A^{j,+}\,,
\end{cases}
\end{align*}
and this implies the claimed decomposition.
\end{proof}

\section{Rectifiability}\label{s:rectif2}

The proof of Theorem~\ref{t:rectif} 
relies on the following estimate.

\begin{prop}\label{p:estim_small_jumps}
Let $m\in B^{1/3}_{3,\infty}(\Omega;\R^2)$ a weak solution of the eikonal equation \eqref{eq:eik} and $\omega$ a minimal  Lagrangian representation of $m$.

For every $\delta\in (0,\frac\pi8)$ and $\alpha\in (0,1)$, we have
\begin{align*}
&
\int_{\Gamma} 
|D\gamma_s|
\big(
\lbrace 
t\in I_\gamma\colon \gamma_x(t) \in \Omega \setminus J_m^\delta, \ 
|D\gamma_s|(\lbrace t\rbrace )\leq \delta\rbrace\big)
\,d\omega(\gamma)
\\
&
\lesssim \delta^{1-\alpha}
\bigg(\bigvee_{\|\Phi\|_{C^{1,\alpha}}\leq 1}|\dv\Phi(m)|\bigg)(\Omega \setminus J_m^\delta)\,,
\end{align*}
where the supremum of measures is taken over all entropies $\Phi\in\widetilde{\mathrm{ENT}}$ such that $\|\Phi\|_{C^{1,\alpha}}\leq 1$.
\end{prop}

\begin{proof}[Proof of Theorem~\ref{t:rectif}]
Let $m\colon \Omega\to\R^2$ a weak solution of the eikonal equation \eqref{eq:eik} such that $m\in B^{1/p}_{p,\infty}(\Omega)$ for some $p\in [1,3)$.
Since $B^{1/p}_{p,\infty}\cap L^\infty \subset B^{1/q}_{q,\infty}$ for all $q \geq p$, we assume without loss of generality that $2<p<3$ and may write $p=2+\alpha$ for some $\alpha\in (0,1)$.
Thanks to Proposition~\ref{p:ent_prod}, the supremum measure appearing in Proposition~\ref{p:estim_small_jumps} is finite.
Letting $\delta\to 0$, we deduce 
\begin{align*}
\int_{\Gamma} 
|D\gamma_s|
\big(
\lbrace 
t\in I_\gamma\colon \gamma_x(t) \in \Omega \setminus J_m, \ 
|D\gamma_s|(\lbrace t\rbrace )=0\rbrace\big)
\,d\omega(\gamma) =0\,.
\end{align*}
In other words, for $\omega$-a.e. $\gamma\in \Gamma$, the measure $|D\gamma_s|$ only has a jump part.
Thanks to the representation formula \eqref{e:repr_entr} and Lemma \ref{l:jumpLE}, this implies that the entropy dissipation is concentrated on the jump set $J_m$ for every entropy $\Phi \in \mathrm{ENT}$.
\end{proof}

\begin{proof}[Proof of Proposition~\ref{p:estim_small_jumps}]
We fix $\delta\in (0,\pi/16)$ and start by defining a family of functions $\psi$ bounded in $C^{0,\alpha}$ and to which Lemma~\ref{l:pos-neg-jumps} can be applied.
We define $\psi_0\in C^{0,1}(\T;\R)$, odd, $\pi$-periodic, even with respect to $\pi/4$, and given on $[0,\pi/4]$ by the nondecreasing piecewise affine function
\begin{align*}
\psi_0(t)=\begin{cases}
t/\delta^{1-\alpha}
 &\text{ for }0\leq t\leq 2\delta,
\\
2\delta^{\alpha} & \text{ for }2\delta < t\leq \pi/4\,.
\end{cases}
\end{align*}
This function $\psi_0$ generates the family $\psi_{\bar s}=\psi_0(\cdot -\bar s)$, for $\bar s\in\T$.
The corresponding entropies $\Phi_{\bar s}=\Phi_{\psi_{\bar s}}$, defined as in \eqref{e:Phipsi}, satisfy the bound
$\|\psi_{\bar s}\|_{C^{1,\alpha}}\lesssim 1$. 

For any $\bar s\in \T$ and any $\gamma\in\Gamma$  we have
\begin{align*}
&
|D\gamma_s|
\big(
\lbrace
t\in I_\gamma\colon 
\gamma(t^-)\text{ and }\gamma(t^+)\in [\bar s-\delta,\bar s +\delta]
\rbrace
\big)
\\
&
\leq \delta^{1-\alpha} |D(\psi_{\bar s}\circ\gamma_s)|
\big(
\lbrace
t\in I_\gamma\colon 
\gamma(t^-)\text{ and }\gamma(t^+)\in [\bar s-\delta,\bar s +\delta]
\rbrace
\big)\,.
\end{align*}
We fix a uniform partition $0=\bar s_1 <\bar s_2 < \cdots <\bar s_N <\bar s_{N+1}=2\pi$ of $[0,2\pi]$ such that
$\delta/2 <  \bar s_{j+1}-\bar s_j \leq \delta$. 
Then we have the covering
\begin{align*}
\T\subset\bigcup_{j=1}^N \bar I_j,\quad \bar I_j = [\bar s_j-\delta,\bar s_j +\delta]\,,
\end{align*}
 each interval $\bar I_j$ intersects at most 4 of the other intervals from this covering,
 and any interval of length $\delta$ is contained in one of these intervals.
Thus we have
\begin{align*}
&
|D\gamma_s|
\big(
\lbrace 
t\in I_\gamma\colon 
|D\gamma_s|(\lbrace t\rbrace )\leq \delta\rbrace\big)
\\
&
\leq
\delta^{1-\alpha} 
\sum_{j=1}^N 
|D(\psi_{\bar s_j}\circ\gamma_s)|
\big(
\lbrace
t\in I_\gamma\colon 
\gamma(t^-)\text{ and }\gamma(t^+)\in \bar I_j
\rbrace
\big)\,.
\end{align*}
Invoking Lemma~\ref{l:small_jumps} below, we see that
\begin{align*}
\lbrace
t\in I_\gamma\colon 
\gamma(t^-)\text{ and }\gamma(t^+)\in \bar I_j
\rbrace
\subset
\lbrace 
t\in I_\gamma\colon \gamma_x(t)\in A_j
\rbrace \cup Z_j\,,
\end{align*}
where $|D\gamma_s|(Z_j)=0$ and
\begin{align*}
A_j
=\Big\lbrace x\in\Omega\colon 
&
\exists\tilde s\in \bar I_j\,,
\;
x\text{ is either }
\\
&
\text{a jump point of }m
\text{ with  normal }\mathbf n\in \lbrace ie^{i\tilde s},e^{i\tilde s}\rbrace \,,
\\
&
\text{or a Lebesgue point of } \mathfrak{s}
\text{ with value }
\mathfrak s (x)\in \lbrace \tilde s, \tilde s + \pi\rbrace
\Big\rbrace\,.
\end{align*}
In the definition above, we refer to the function $\mathfrak s$ defined in Theorem \ref{t:kin} and we consider Lebesgue points with respect to the measure $\nu \llcorner (\Omega \setminus J_m)$.
The finite intersection property of the intervals $\bar I_j$ implies that each set $A_j$ intersects at most 16 of the other sets $A_k$, $k\neq j$.
Moreover, thanks to Lemma~\ref{l:abs_no_jumps} and Lemma \ref{l:pos-neg-jumps}, we have
\begin{align*}
|\dv\Phi_{\bar s_j}(m)|(A_j \setminus J_m^\delta)
&
=\int_{\Gamma} |D(\psi_{\bar s_j}\circ\gamma_s)|(\lbrace \gamma_x\in A_j\setminus J_m^\delta\rbrace )\, d\omega(\gamma)\,.
\end{align*}
Gathering the above properties, we deduce
\begin{align*}
&
\int_{\Gamma}
|D\gamma_s|
\big(
\lbrace 
t\in I_\gamma\colon \gamma_x(t) \in \Omega \setminus J_m^\delta, \ 
|D\gamma_s|(\lbrace t\rbrace )\leq \delta\rbrace\big)
\,
d\omega(\gamma)
\\
&
\lesssim
\delta^{1-\alpha}\sum_{j=1}^N |\dv\Phi_{\bar s_j}(m)|(A_j\setminus J_m^\delta)\,.
\end{align*}
And the finite intersection property of the $A_j$'s, together with the boundedness of the family $\Phi_{\bar s}$ in $C^{1,\alpha}$ implies that the last sum is
controlled by the supremum measure appearing in Proposition~\ref{p:estim_small_jumps}.
\end{proof}

\begin{lem}\label{l:small_jumps}
Let $m\in B^{1/3}_{3,\infty}(\Omega;\R^2)$ be a weak solution of the eikonal equation \eqref{eq:eik} and $\omega$ be a minimal Lagrangian representation of $m$. Then for $\omega$-a.e. $\gamma\in\Gamma$ we have the following.

For any $\delta>0$, any $\bar s\in\T$ and $|D\gamma_s|$-a.e. $t\in I_\gamma$, 
if 
\begin{align*}
\gamma(t^-)\text{ and }\gamma(t^+)\in [\bar s - \delta,\bar s +\delta]\,,\qquad  \mbox{and}\qquad \gamma_x(t) \in \Omega \setminus J_m^\delta,
\end{align*}
 then there exists $\tilde s\in [\bar s - \delta,\bar s +\delta]$ such that
$x=\gamma_x(t)$ is either a jump point of $m$ with normal $n\in \lbrace ie^{i\tilde s},e^{i\tilde s}\rbrace$ or a Lebesgue point of $\mathfrak s$ with value $\mathfrak s(x)\in\lbrace \tilde s, \tilde s + \pi \rbrace$.
\end{lem}

\begin{proof}[Proof of Lemma~\ref{l:small_jumps}]
It follows from Theorem \ref{t:lagr} that for $\omega$-a.e. $\gamma \in \Gamma$ the following holds: for $|D\gamma_s|$-a.e. $t \in I_\gamma$
\[
\gamma_s(t+), \gamma_s(t-) \in \mathrm{supp}\, (\sigma_{\min})_{\gamma_x(t)}.
\]
Recall that $\nu_{\min}(A)=0$ implies $\int_{\Gamma}|D\gamma_s|(\{\gamma_x \in A\}) d \omega (\gamma)=0$.
Since for $\nu_{\min}$-a.e. $x \in J _m\setminus J_m^\delta$ the support of $(\sigma_{\min})_x$ is contained in $[\tilde s - \frac\pi2 - \delta,  \tilde s -\frac\pi2 + \delta] \cup [\tilde s + \frac\pi2 - \delta,  \tilde s +\frac\pi2 + \delta] $ where $\tilde s\in \T$ is such that $e^{i\tilde s} = {\bf n}(x)$ and for $\nu_{\min}$-a.e. $x \in \Omega \setminus J_m$ the support of $(\sigma_{\min})_x$ is contained in $\{ \mathfrak s(x)-\frac \pi2,  \mathfrak s(x)+\frac \pi2\}$, then the claim follows.
\end{proof}

\section*{Acknowledgments} 
XL is supported by the ANR project ANR-22-CE40-0006.
EM is member of the GNAMPA group of INDAM and acknowledges the hospitality of the Institut de Math\'ematiques de Toulouse, where part of the work has been done.

\small

\bibliographystyle{acm}

\end{document}